\documentclass[12pt,a4paper]{amsart}

\usepackage[utf8]{inputenc}

\usepackage[T1]{fontenc}         

\usepackage{amsmath}             
\usepackage{amssymb}            
\usepackage{amsfonts}             
\usepackage{mathrsfs}
                                 
\usepackage{latexsym}            
\usepackage{amsthm}

\usepackage{url}

\newcommand{\field}[1]{\mathbb{#1}}    
\newcommand{\N}{\field{N}}             
\newcommand{\Z}{\field{Z}}            
            
\newcommand{\R}{\field{R}}            
           
\newcommand{\Hb}{\field{H}}			 
\newcommand{\mc}[1]{\mathcal{#1}}

\DeclareMathOperator{\diam}{\mathrm{diam}}
\DeclareMathOperator{\capa}{\mathrm{Cap}}

\theoremstyle{plain}
\newtheorem{theorem}{Theorem}[section]    
\newtheorem{lemma}[theorem]{Lemma}      
      
\newtheorem{prop}[theorem]{Proposition}
\newtheorem{sublemma}{Sublemma}

\theoremstyle{definition}

\theoremstyle{remark}
\newtheorem{remark}[theorem]{Remark}

\usepackage{graphicx}
\graphicspath{ {./images/} }

\usepackage{blindtext}

\begin{document}
\title[Limsup sets of rectangles in Heisenberg groups]{Hausdorff dimension of limsup sets of isotropic rectangles in Heisenberg groups}
\author{Markus Myllyoja}
\date{}

\thanks{This project has been supported by the Academy of Finland, project grant No. 318217, and by the University of Oulu Graduate School. I also thank my doctoral supervisors Esa Järvenpää and Maarit Järvenpää for their guidance throughout this project.}

\begin{abstract}
A formula for the Hausdorff dimension of typical limsup sets generated by randomly distributed isotropic rectangles in Heisenberg groups is derived in terms of directed singular value functions.
\end{abstract}

\maketitle

\section{Introduction}
In this paper we study the Hausdorff dimension of a typical limsup set generated by rectangles, where the centers of the rectangles are distributed randomly according to the Lebesgue measure. A motivation for our studies is to find counterparts in the Heisenberg setting for results known about random limsup sets in Euclidean spaces.

Recently, dimensional properties in Heisenberg groups have been studied widely. Usually the objective of these studies is to consider some results known in Euclidean spaces and find the correct analogues in the Heisenberg setting. Since the relation between Hausdorff dimension with respect to the Euclidean and the Heisenberg metrics is not a trivial one \cite{BRSC2003, BT2005}, the Euclidean dimension results do not directly translate to results in Heisenberg groups and more careful consideration is required.

Hausdorff dimensions of self-similar and self-affine sets in the Heisenberg setting have been studied, for example, in \cite{BBMT2010,BT2005,BTW2009,CTU2020}. In \cite{BDCFM2013,BMFT,FH2016,H2014}, the authors study the Hausdorff dimensions of projections and slices in Heisenberg groups. The two types of natural projections and slices in Heisenberg groups are the horizontal and vertical ones. The vertical projections are not quite as well-behaved as the horizontal projections in the sense that they are, for example, not Lipschitz continuous. This means that methods employed in the Euclidean projection theorems cannot be used in the Heisenberg setting and the theorems themselves turn out to have a different character than their Euclidean counterparts.

Our focus is on the dimensional properties of random limsup sets in Heisenberg groups. Given a space $X$ and a sequence $(A_n)$ of subsets of $X$, the limsup set generated by the sequence $(A_n)$ is the set of points in $X$ which belong to infinitely many of the sets $A_n$, that is,
\begin{align*}
\limsup_{n\to \infty}A_n=\bigcap_{n=1}^{\infty}\bigcup_{k=n}^{\infty}A_k.
\end{align*}
The study of limsup sets has a long history, and can be traced back at least as far as the Borel-Cantelli lemma. These limsup sets are essential objects in the study of Besicovitch-Eggleston sets related to $k$-adic expansions of real numbers \cite{B1935,E1949}, and in Diophantine approximation \cite{J1932,K1926}. For developments in Diophantine approximation in Heisenberg groups, see \cite{SV2017,V2016,Z2019}.

Dimensional properties of random limsup sets have been studied since \cite{FW2004}. There the authors considered limsup sets of arcs whose centers are randomly placed on the unit circle where the randomness is determined by the Lebesgue measure. They obtained a formula for the Hausdorff dimension of a typical limsup set generated by arcs. 

Since \cite{FW2004}, there have been many developments in the study of dimensional properties of random limsup sets in metric spaces.
For various results, see  \cite{D2010,EP2018,FST2013,FJJS2018,JJKLS2014,JJKLSX2017,P2015,P2019,S2018}. 
Together, the dimension results in these papers yield the almost-sure values of Hausdorff dimensions of random limsup sets in varying scenarios depending on properties of the ambient space, the sets generating the limsup sets and the measure according to which the sets are distributed.  

If the underlying space is a Riemann manifold and the driving measure is the Lebesgue measure, then the generating sets can be any Lebesgue measurable sets satisfying some necessary density conditions. If the underlying space is the Euclidean or the symbolic space and the generating sets are balls, then the restrictions imposed on the measure can be relaxed a bit, although some special properties, like the measure being a Gibbs measure, are required. The underlying space can be taken to be any Ahlfors regular metric space provided that the generating sets are balls and the randomness is determined by the natural measure.

In \cite{EJJS2018}, the authors consider random limsup sets generated by rectangles in products of Ahlfors regular metric spaces. There, the Lipschitz continuity of projections plays an essential role, hence the same methods are not readily available when considering rectangles in Heisenberg groups and some new machinery is required.

In the paper \cite{ekstrometal}, the authors obtained an almost-sure formula for the Hausdorff dimension of a random limsup set generated by a certain type of rectangles in the first Heisenberg group. To accomplish this, they generalized certain Euclidean results to the context of unimodular groups (see Theorem~\ref{theorem1.3}). They then carried out calculations specific to these types of rectangles and the first Heisenberg group to complete the proof.

With the definition of rectangles given in \cite{ekstrometal}, the dimension formula extends to higher-dimensional Heisenberg groups in a rather straightforward manner (see Remark~\ref{remark:otherrectangles} below). However, in Heisenberg groups there are also other natural ways of defining rectangles and the purpose of this paper is to obtain a dimension formula that is valid for limsup sets generated by these types of rectangles.

The paper is organized as follows:
First, we fix some notations and introduce the required concepts in order to present our main result (see Theorem~\ref{theorem1.1} below). In Section~\ref{section:contentandenergy} we establish upper bounds for the Hausdorff content and lower bounds for the capacity of rectangles. In Section~\ref{section:proofofmainresult}, we combine the results of Section~\ref{section:contentandenergy} with a theorem proved in \cite{ekstrometal} to prove  Theorem~\ref{theorem1.1}. We end the paper with a brief comment on the main result in \cite{ekstrometal} and its generalization to higher dimensional Heisenberg groups.

\section{Notation and Main Result}
In a metric space $(X,d)$, we write $B(x,r)$ and $\overline{B}(x,r)$ for the open and closed balls with radius $r$ and center $x\in X$, respectively. Throughout this paper,  $a\wedge b$ denotes the minimum of two real numbers $a$ and $b$.
 For $x,y\in \R^m$ we denote the Euclidean norm of $x$ simply by $|x|$ and we write $\langle x,y\rangle$ for the Euclidean inner product of $x$ and $y$. The dimension $m$ of the Euclidean space in question will be clear from context.

Before stating our main result, we introduce the Heisenberg groups and discuss their properties which we will need in our analysis. For more details on general properties of Heisenberg groups we refer the reader to \cite{BMFT} and \cite{MSSC}.

The Heisenberg group $\Hb:=\Hb^n$ is the set $\R^{2n+1}$ equipped with the group operation 
\begin{align*}
p\ast p'&=\left( x+x',y+y',z+z'+2\sigma((x,y),(x',y'))\right) \\
&=\left( x+x',y+y',z+z'+2\sum_{i=1}^{n}\left( x_iy'_i-y_ix'_i \right)\right)\\
&=\left( x+x',y+y',z+z'+2\left( \langle x,y'\rangle-\langle y,x'\rangle \right)\right)
\end{align*}
 where $p=(x,y,z), p'=(x',y',z')\in \R^n\times\R^n\times \R$ and $\sigma$  is the standard symplectic form 
on $\R^{2n}$ given by the formula
 \begin{align*}
 \sigma((x,y),(x',y'))=\sum_{i=1}^{n}\left( x_iy'_i-y_ix'_i \right) .
 \end{align*}
 The neutral element of $\Hb$ is $(0,0,0)$ and the inverse of an element $p=(x,y,z)$ is $p^{-1}=(-x,-y,-z)$. There is a norm on $\Hb$ given by
$$\Vert p\Vert_{\Hb}=\left( |(x,y)|^{4}+z^2 \right)^{1/4}$$
which determines a left-invariant metric via the formula
\begin{align*}
d_{\Hb}(p,p')&=\Vert p^{-1}p'\Vert_{\Hb}\\
&=\Big( |(x',y')-(x,y) |^{4}+\big( z'-z-2(\langle x,y' \rangle -\langle y,x' \rangle) \big)^{2} \Big)^{1/4} .
\end{align*}
The Haar measure on $\Hb$ is the Lebesgue measure, which is invariant under both left and right translations so that $\Hb$ is a unimodular group. In the metric $d_{\Hb}$, the Lebesgue measure of $B(p,r)$ is proportional to $r^{2n+2}$ and thus the metric space $(\Hb,d_{\Hb})$ has Hausdorff dimension $2n+2$.
The vertical line through the origin is $$L(0)=\{ p\in \Hb\mid x=y=0\},$$ and the vertical line through $p$ is the set 
$$L(p):=pL(0)=\{ (x,y,z')\in \Hb\mid z'\in \R\}.$$ Moreover, the horizontal plane through the origin is 
$$H(0)=\{ p\in \Hb\mid z=0\} $$
and the horizontal plane through $p$ is defined as
$$H(p):=pH(0)=\{ p'\in \Hb\mid z'=z+2(\langle x,y' \rangle -\langle y,x' \rangle)\} .$$
Note that $d_{\Hb}(p,p')\geq |(x,y)-(x',y') |$ for every $p$ and $p'$ in $\Hb$, with an equality if and only if $p'\in H(p)$. 

In the following we will work with homogeneous subgroups of $\Hb$. These are the closed subgroups of $\Hb$ which are invariant under the intrinsic dilations $(x,y,z)\mapsto (sx,sy,s^2z)$, $s>0$. There are two kinds of homogeneous subgroups of $\Hb$, namely, the horizontal subgroups and vertical subgroups. A horizontal subgroup $\mathbb{V}$ of $\Hb$ is of the form $\mathbb{V}=V\times\{0\}$, where $V$ is an isotropic subspace of $\R^{2n}$, that is, $V\subseteq \R^{2n}$ is a linear subspace with the property that $\sigma_{|_{V}}=0$. A vertical subgroup $\mathbb{V}^{\perp}\subseteq \Hb$ is a subgroup  of the form $\mathbb{V}^{\perp}=V^{\perp}\times\R$, where $V^{\perp}\subseteq \R^{2n}$ is the orthogonal complement of an isotropic subspace $V$. The dimension (as a vector space) of a nontrivial isotropic subspace $V\subseteq \R^{2n}$ can be any value from the set $\{1,\ldots,n\}$.

Given any linear subspace $W\subseteq \R^{2n}$, we denote the orthogonal projection onto $W$ by $P_{W}\colon \R^{2n}\to W$. Each isotropic subspace $V\subseteq \R^{2n}$ yields a semidirect group splitting of $\Hb$ in two ways. Firstly, we obtain the Heisenberg group as the semidirect product $\Hb=\mathbb{V}^{\perp}\rtimes \mathbb{V}$ since each point $p\in \Hb$ can be uniquely written as 
$$p=\left(P_{V^{\perp}}(x,y),z-2\sigma(P_{V^{\perp}}(x,y),P_{V}(x,y))\right)\ast(P_{V}(x,y),0) $$
where $\left(P_{V^{\perp}}(x,y),z-2\sigma(P_{V^{\perp}}(x,y),P_{V}(x,y))\right)\in \mathbb{V}^{\perp}$ and $(P_{V}(x,y),0)\in \mathbb{V}$.
This gives rise to well-defined horizontal and vertical projections
\begin{align*}
P_{\mathbb{V}}\colon \Hb\to \mathbb{V},  P_{\mathbb{V}}(p)=(P_{V}(x,y),0)
\end{align*} 
and 
\begin{align*}
P_{\mathbb{V}^{\perp}}\colon \Hb\to \mathbb{V}^{\perp},  P_{\mathbb{V}^{\perp}}(p)=\left(P_{V^{\perp}}(x,y),z-2\sigma(P_{V^{\perp}}(x,y),P_{V}(x,y))\right).
\end{align*}
On the other hand, we have that $\Hb=\mathbb{V}\rtimes \mathbb{V}^{\perp}$, since each point $p$ has a unique representation
$$p=(P_{V}(x,y),0)\ast\left(P_{V^{\perp}}(x,y),z+2\sigma(P_{V^{\perp}}(x,y),P_{V}(x,y))\right) $$
where $$Q_{\mathbb{V}}(p):=(P_{V}(x,y),0)\in \mathbb{V}$$
and
$$Q_{\mathbb{V}^{\perp}}(p):=\left(P_{V^{\perp}}(x,y),z+2\sigma(P_{V^{\perp}}(x,y),P_{V}(x,y))\right)\in \mathbb{V}^{\perp}$$
determine the horizontal and vertical projections (respectively) with respect to this group splitting.
The horizontal projections are Lipschitz continuous and also group homomorphisms, whereas the vertical projections are neither.

For $1\leq d\leq n$ we write 
$$G_{h}(n,d)=\{ V\subseteq \R^{2n}\mid V \text{ is a } d\text{-dimensional isotropic subspace of } \R^{2n}\}. $$
The set $G_{h}(n,d)$ is called the isotropic Grassmannian and it is a submanifold of the usual Grassmannian manifold $G(2n,d)$ consisting of all $d$-dimensional linear subspaces of $\R^{2n}$.

A matrix $M\in M(2n,\R)$ is called symplectic, if it preserves the symplectic form $\sigma$. A matrix which is both symplectic and orthogonal is called orthosymplectic. Write 
$$U(n)=\{ M\in M(2n,\R)\mid M \text{ is orthosymplectic}\} .$$
The set $U(n)$ with matrix multiplication forms a group which acts transitively on $G_{h}(n,d)$ (for the proof of this fact see \cite[Lemma~2.1]{BMFT}). Note that since any $U\in U(n)$ preserves both the form $\sigma$ and the Euclidean inner product, the induced map $\widetilde{U}\colon \Hb\to\Hb$, $(x,y,z)\mapsto (U(x,y),z)$
 preserves both the Heisenberg norm and the metric.

Given $V\in G_{h}(n,d)$ and
a pair of positive numbers $r=(r_1,r_2)$, we define the closed rectangle of type $1$ centered at the origin to be the set
\begin{align*}
&\overline{R}_{1,V}(0,r):=\left\{ p'\mid \Vert P_{\mathbb{V}}(p')\Vert_{\Hb}\leq r_1,\ \Vert P_{\mathbb{V}^{\perp}}(p')\Vert_{\Hb}\leq r_2 \right\}, 
\end{align*}
and the closed rectangle of type $1$ centred at $p$ is defined by
\begin{align*}
\overline{R}_{1,V}(p,r)&:=p\overline{R}_{1,V}(0,r)\\
&=\left\{ p'\mid \Vert P_{\mathbb{V}}(p^{-1}p')\Vert_{\Hb}\leq r_1,\ \Vert P_{\mathbb{V}^{\perp}}(p^{-1}p')\Vert_{\Hb}\leq r_2 \right\}.
\end{align*}

We define rectangles with respect to the group splitting $\Hb=\mathbb{V}\rtimes \mathbb{V}^{\perp}$ similarly with $Q$ in place of $P$. More precisely, we let
\begin{align*}
&\overline{R}_{2,V}(0,r):=\left\{ p'\mid \Vert Q_{\mathbb{V}}(p')\Vert_{\Hb}\leq r_1,\ \Vert Q_{\mathbb{V}^{\perp}}(p')\Vert_{\Hb}\leq r_2 \right\} 
\end{align*}
be the closed rectangle of type $2$ centred at the origin and we define
\begin{align*}
\overline{R}_{2,V}(p,r)&:=p\overline{R}_{2,V}(0,r)\\
&=\left\{ p'\mid \Vert Q_{\mathbb{V}}(p^{-1}p')\Vert_{\Hb}\leq r_1,\ \Vert Q_{\mathbb{V}^{\perp}}(p^{-1}p')\Vert_{\Hb}\leq r_2 \right\}
\end{align*}
to be the closed rectangle of type $2$ centred at $p$.

By the left-invariance of the metric and the relations
$$pP_{\mathbb{V}^{\perp}}(p^{-1}p')P_{\mathbb{V}}(p^{-1}p')=p'=pQ_{\mathbb{V}}(p^{-1}p')Q_{\mathbb{V}^{\perp}}(p^{-1}p') $$
we obtain the following interpretations for the rectangles. 
The rectangle $\overline{R}_{1,V}(p,r)$ is the set of such points that can be reached from the center $p$ by first moving "vertically"\ in $p\mathbb{V}^{\perp}$ a distance at most $r_2$ to end up at a point, say $q\in p\mathbb{V}^{\perp}$, and then moving "horizontally"\ in $q\mathbb{V}$ a distance at most $r_1$. For the rectangle $\overline{R}_{2,V}(p,r)$ we first move from the center "horizontally"\ in $p\mathbb{V}$ a distance at most $r_1$ and then "vertically"\ in $q\mathbb{V}^{\perp}$ a distance at most $r_2$.

Figures \ref{figure1} and \ref{figure2} shed light into what the rectangles look like in the first Heisenberg group when $\mathbb{V}$ is the $x$-axis. For explicit formulas defining the rectangles in this case, see the proof of Proposition~\ref{prop:upperboundrectangle}.

\begin{figure}[h]

\includegraphics[scale=0.1985]{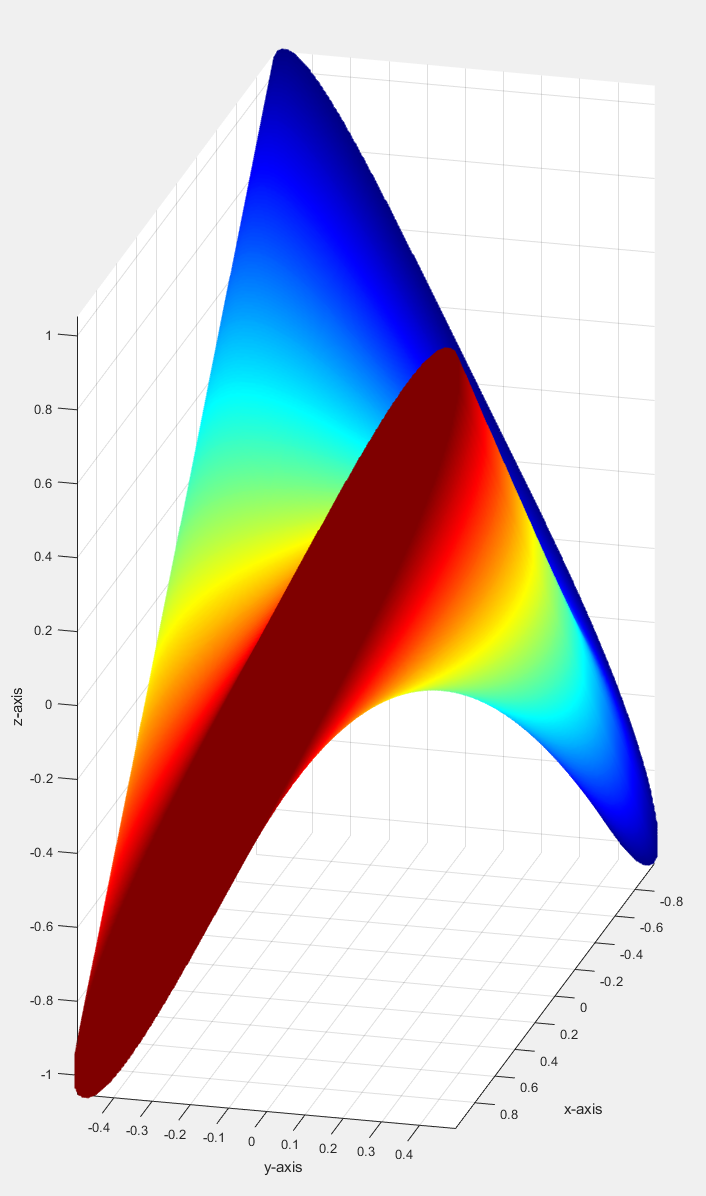}\includegraphics[scale=0.20]{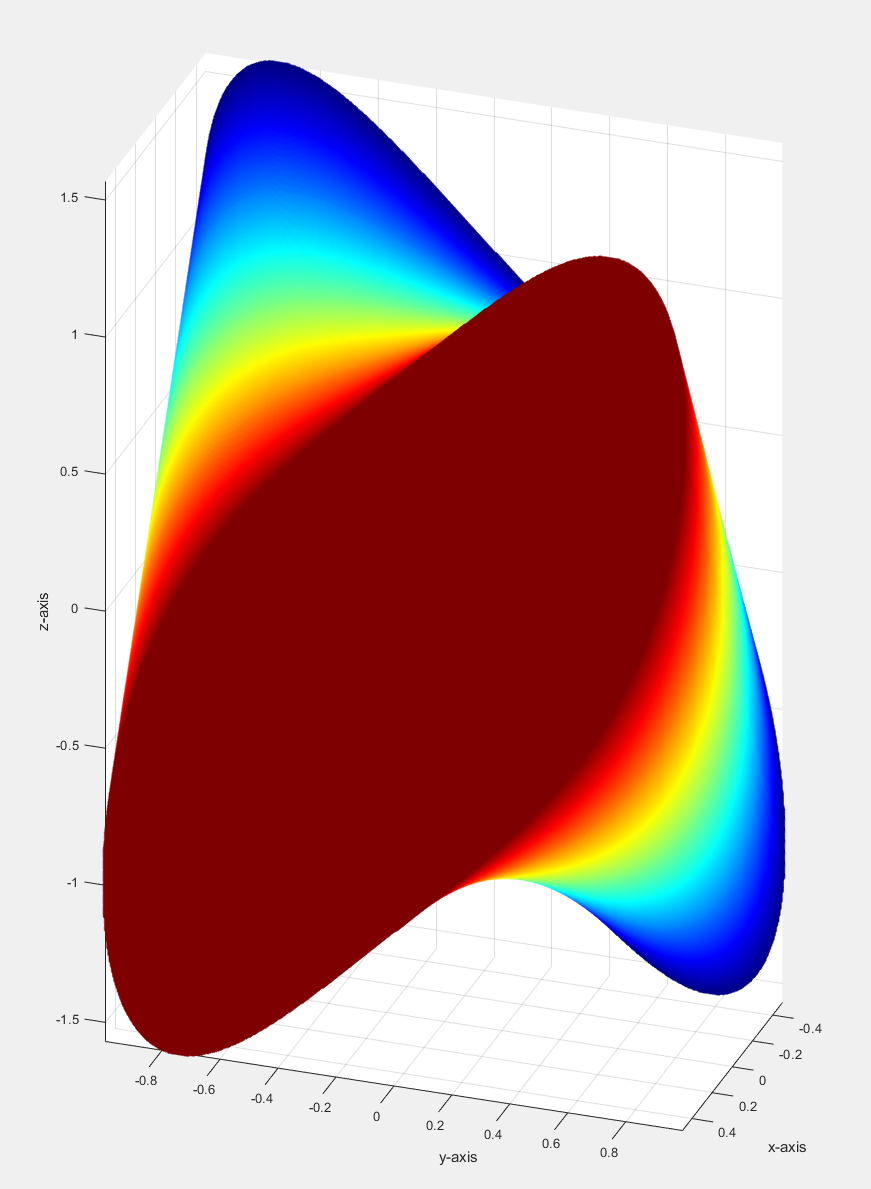}
\caption{Rectangles of type $2$ centered at the origin with $r=(1/2,1)$ and $r=(1,1/2)$, respectively.}
\label{figure1} 
\end{figure}

\begin{figure}[h]

\includegraphics[scale=0.1965]{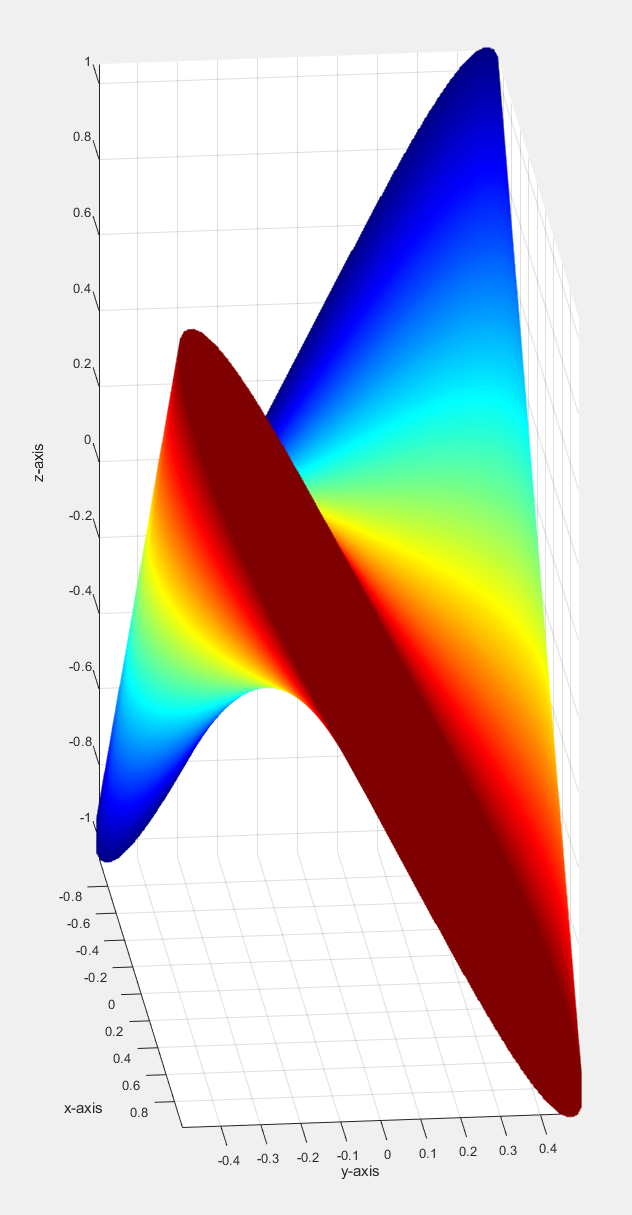}\includegraphics[scale=0.20]{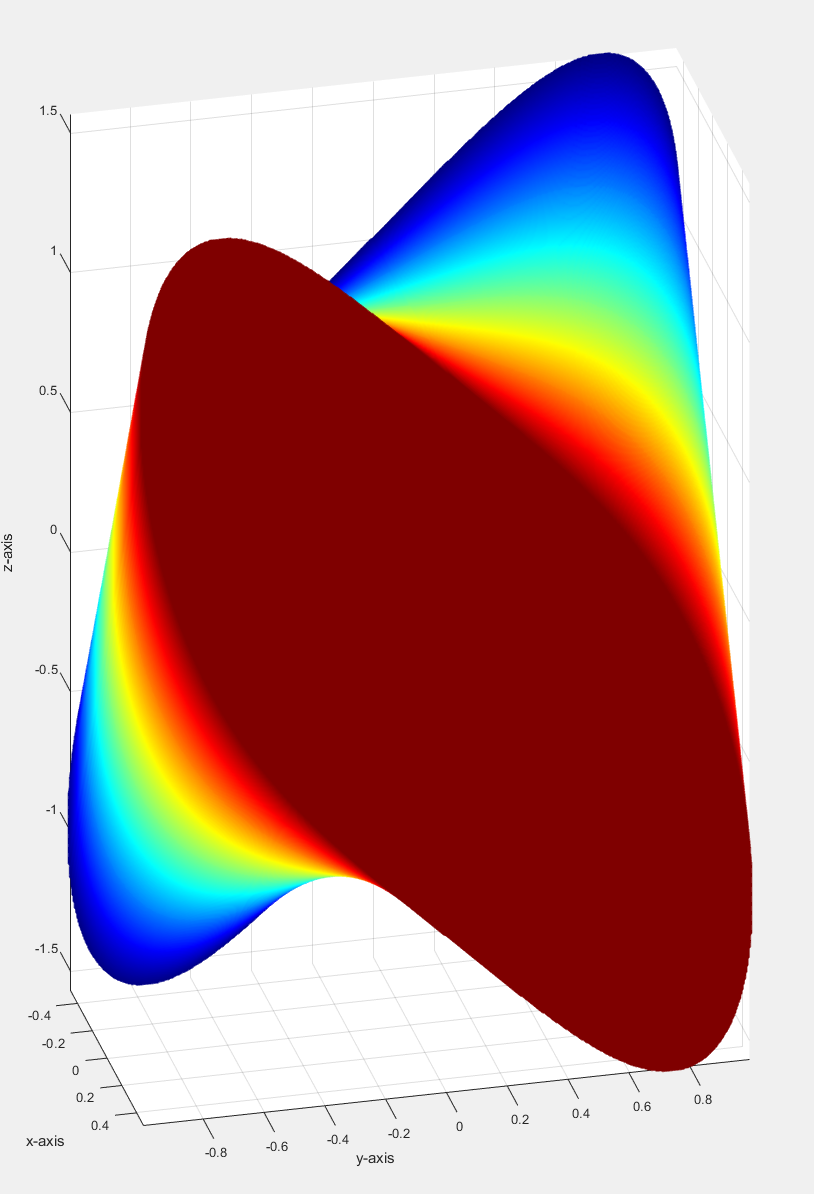}
\caption{Rectangles of type $1$ centered at the origin with $r=(1/2,1)$ and $r=(1,1/2)$, respectively.}
\label{figure2}
\end{figure}

While the rectangles of type $1$ and type $2$ look rather similar and have, for example, the same Lebesgue measure, they are quite different objects in the eyes of the Heisenberg metric. Indeed, for type $2$ rectangles depicted in the above figures, the set of points sharing the same $x$-coordinate is just the intersection $\overline{B}((x,0,0),r_2)\cap (\{ x\}\times \R^2)$. On the other hand, for type $1$ rectangles this set of points is  smaller than $r_2$ in some directions and much larger in some other directions provided that the $x$-coordinate is large compared to $r_2$. Because of this, the analysis on type $1$ rectangles is  (in the case $r_1> r_2$) somewhat more complicated as we will see later in the proofs of our results.

Figure \ref{figure3} demonstrates the effect of translation on rectangles.

\begin{figure}[h]

\includegraphics[scale=0.193]{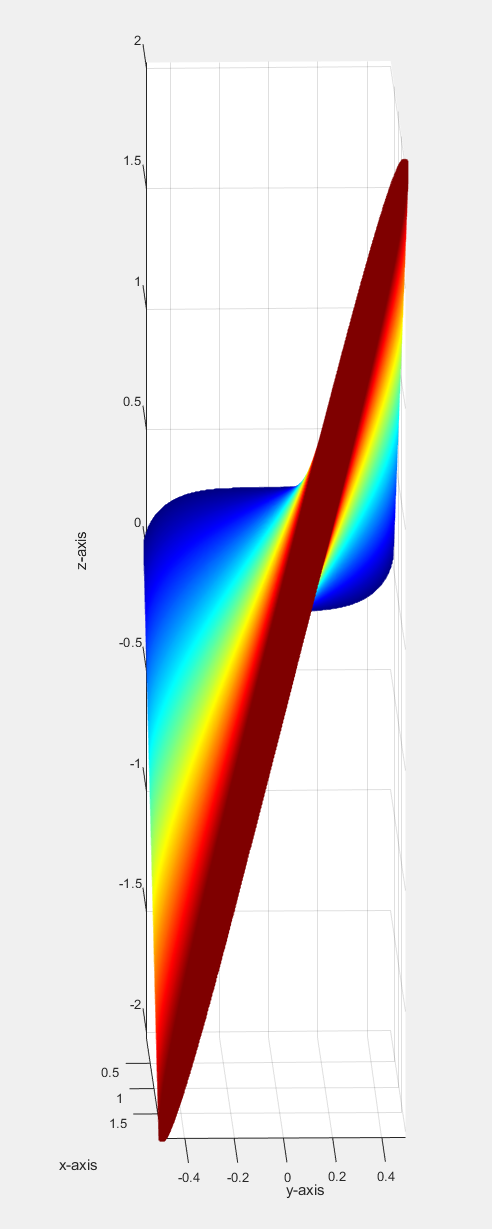}\includegraphics[scale=0.1956]{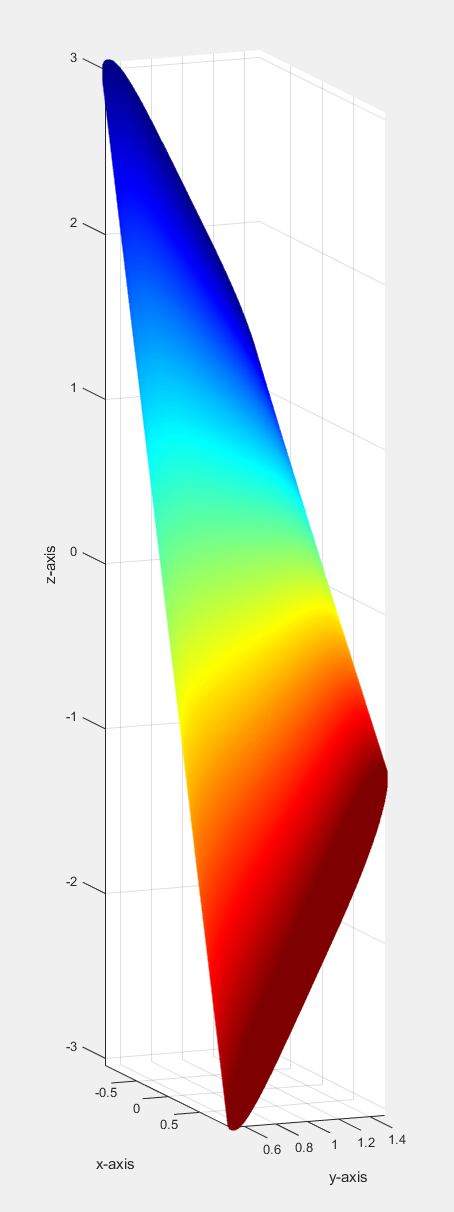}\includegraphics[scale=0.2]{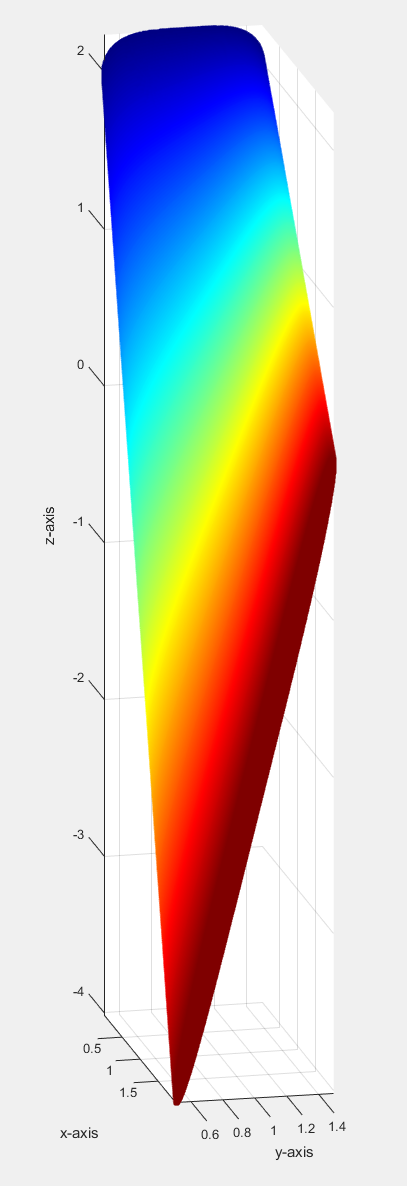}
\caption{The rectangle $\overline{R}_{2,V}(0,(1/2,1))$ from Figure \ref{figure1} translated to the centers $(1,0,0)$, $(0,1,0)$ and $(1,1,0)$, respectively.}
\label{figure3}
\end{figure}

It is already evident from the rectangles above that the shape of a translated rectangle depends greatly on the center point $p$. When considering a rectangle centered at some arbitrary point, the following formulas are useful. For any $p\in \Hb$, it is true that
\begin{align*}
p\mathbb{V}&=\{ p'\mid p^{-1}p'\in \mathbb{V}\}\\ &=\{ p'\mid (x',y')-(x,y)\in V, z'=z+2(\langle x,y' \rangle-\langle y,x' \rangle)\}
\end{align*} 
and
\begin{align*}
p\mathbb{V}^{\perp}&=\{ p'\mid p^{-1}p'\in \mathbb{V}^{\perp}\}=\{ p'\mid (x',y')-(x,y)\in V^{\perp} \}.
\end{align*}

For an integer $d\in \{1,\ldots,n\}$, a pair of positive numbers $r=(r_1,r_2)$ and $t\in [0,2n+2]$ we define the directed singular value functions $\Phi_{1}^t$ and $\Phi_{2}^t$ as follows: If $r_1\leq r_2$, let
\begin{align*}
\Phi_{1}^t(r,d)=\Phi_{2}^t(r,d)=\begin{cases}
r_2^t, &\text{ if } t\in [0,2n+2-d],\\
r_1^{t+d-2(n+1)}r_2^{2(n+1)-d}, &\text{ if } t\in [2n+2-d,2n+2].
\end{cases}
\end{align*}
If $r_1\geq r_2$ let
\begin{align*}
\Phi_{1}^t(r,d)=\begin{cases}
r_1^t, &\text{ if } t\in [0,d],\\
r_1^{\frac{t+d}{2}}r_2^{\frac{t-d}{2}}, &\text{ if } t\in [d,d+2],\\
r_1^{d+1}r_2^{t-d-1}, &\text{ if } t\in ]d+2,2n+1[,\\
r_1^{2n+2+d-t}r_2^{2t-(2n+2+d)}, &\text{ if } t\in [2n+1,2n+2],
\end{cases}
\end{align*}
and let
\begin{align*}
\Phi_{2}^t(r,d)=\begin{cases}
r_1^t, &\text{ if } t\in [0,d],\\
r_1^{d}r_2^{t-d}, &\text{ if } t\in [d,2n+2].
\end{cases}
\end{align*}

Denote by $\lambda$ the Lebesgue measure on $\Hb$ and let $W\subseteq \Hb$ be an open bounded subset of $\Hb$. Define the probability space $(\Omega, \mathbb{P})$ by $\Omega=\Hb^{\N}$ and $\mathbb{P}=(\lambda_W)^{\N}$, where $\lambda_{W}=\lambda(W)^{-1}\lambda_{|_{W}}$. 

For a sequence $(d_k)\in\{ 1,\ldots,n\}^{\N}$, a sequence $\underline{V}:=(V_k)$ of isotropic subspaces $V_k\in G_{h}(n,d_k)$, a sequence $\mathbf{j}:=(j_k)\in \{1,2\}^{\N}$ of types, a bounded sequence $\underline{r}:=(r_k)$ of pairs of positive numbers and $\omega=(\omega_1,\omega_2,\ldots)\in \Omega$, let $$E_{\mathbf{j},\underline{V},\underline{r}}(\omega)=\limsup_{k\to \infty}\overline{R}_{j_k,V_k}(\omega_k,r_k).$$
We are now ready to state our main theorem.
\begin{theorem}\label{theorem1.1}
Given sequences $(d_k)$, $\underline{V}=(V_k)$, $\mathbf{j}=(j_k)$ and $\underline{r}=(r_k)$ as above,
$$\dim_{\mathrm{H}} E_{\mathbf{j},\underline{V},\underline{r}}(\omega)=\inf \left\{ t\ \Big| \ \sum_{k}\Phi^{t}_{j_k}(r_k,d_k)<\infty \right\} \wedge (2n+2)$$
for $\mathbb{P}$-almost every $\omega\in\Omega$.
\end{theorem}

\begin{remark}
\begin{enumerate}
\item[a)] If $r_1=r_2=\rho$, then $\Phi_{j}^{t}(r,d)=\rho^{t}$ for every $j,t$ and $d$. In such a case the rectangles are comparable to balls in the sense that 
$$\overline{B}(0,c_{j,d}^{-1}\rho)\subseteq \overline{R}_{j,V}(0,r)\subseteq  \overline{B}(0,c_{j,d}\rho) $$
for some constant $c_{j,d}>1$ depending on $j$ and $d$ but independent of $\rho$ and $V\in G_{h}(n,d)$. If the sequence $(r_k)$ of pairs of positive numbers is contained in the diagonal $\{ (x,x)\mid x>0\}$, then Theorem~\ref{theorem1.1} reduces to the known formula for limsup sets of balls in Ahlfors regular metric spaces (for details, see \cite[Theorem~2.1]{JJKLSX2017} and \cite[Proposition~4.7]{JJKLS2014}). 

\item[b)] We excluded the endpoints $d+2$ and $2n+1$ in the definition of $\Phi^{t}_{1}(r,d)$ for $r_1\geq r_2$ to emphasize the fact that in the first Heisenberg group $\Hb^{1}$ the piece $r_1^{d+1}r_2^{t-d-1}$ does not exist. Indeed, if $n=1$ then automatically $d=1$ (hence $d+2=2n+1$) and for $r_1\geq r_2$ the definition of $\Phi^{t}_{1}(r,1)$ reduces to
\begin{align*}
\Phi^{t}_{1}(r,1)=\begin{cases}
r_1^t, &\text{ if } t\in [0,1],\\
r_1^{\frac{t+1}{2}}r_2^{\frac{t-1}{2}}, &\text{ if } t\in [1,3],\\
r_1^{5-t}r_2^{2t-5}, &\text{ if } t\in [3,4].
\end{cases}
\end{align*}
\end{enumerate}
\end{remark}
\section{Hausdorff content and energy of rectangles}\label{section:contentandenergy}
Our strategy in the proof of Theorem~\ref{theorem1.1} will be as follows: In Proposition~\ref{prop:upperboundrectangle}, we obtain deterministic upper bounds for the Hausdorff content of rectangles, which we will later use to obtain an upper bound for the dimension of a limsup set generated by rectangles. In Proposition~\ref{prop:lowerboundrectangle}, we obtain deterministic lower bounds for the capacity of rectangles. With the help of Theorem \ref{theorem1.3}, these lower bounds then turn into lower bounds for the dimension in Theorem~\ref{theorem1.1}. Before establishing the upper and lower bounds, we recall the relevant definitions. 

The $t$-dimensional Hausdorff content (with respect to the Heisenberg metric) of a set $A\subseteq \Hb$ is defined by
\begin{align*}
\mc{H}_{\infty}^{t}(A)&=\inf \left\{\sum_{k=1}^{\infty} \diam(C_k)^{t}\ \Big| \ A\subseteq \bigcup_{k=1}^{\infty}C_k \right\},
\end{align*}
where the diameter of a set $C\subseteq \Hb$ in the Heisenberg metric is denoted by $\diam(C)$.
For $t>0$ and a Borel measure $\mu$ on $\Hb$, the $t$-energy of $\mu$ is defined by
\begin{align*}
I_t(\mu)&=\int \int d_{\Hb}(p,q)^{-t}\, d\mu(p)\,d\mu(q).
\end{align*}
We define the $t$-energy of a Borel set $A\subseteq \Hb$ to be the quantity $I_t(A)=I_t(\lambda_{|_{A}})$.

The $t$-capacity of a Borel set $A\subseteq \Hb$ is defined by
\begin{align*}
\capa_{t}(A)&=\sup \left\{ I_t(\mu)^{-1}\mid \mu \in \mc{P}(A)  \right\},
\end{align*}
where $\mc{P}(A)$ denotes the space of Borel probability measures on $\Hb$ which are supported on $A$.

We have now defined the necessary concepts and we are ready to proceed with the proof.
 
\begin{lemma}\label{lemma:hausdorffcontentofrectangleindependentofisotropicsubspace}
Let $d\in \{1,\ldots,n\} $ and let $V,V'\in G_{h}(n,d)$. Then for every $j$, $t$ and $r$ it is true that
$$\mc{H}_{\infty}^{t}\left( \overline{R}_{j,V}(0,r) \right)=\mc{H}_{\infty}^{t}\left( \overline{R}_{j,V'}(0,r) \right).$$
\end{lemma}

\begin{proof}
We will first verify the claim for type $1$ rectangles. 
Since $U(n)$ acts transitively on $G_{h}(n,d)$, there exists $U\in U(n)$ such that $U(V')=V$, which implies that also $U(V'^{\perp})=V^{\perp}$. Hence for any point $p=(x,y,z)\in \Hb$ we have that $$U(P_{V'}(x,y))=P_{V}(U(x,y)) \text{ and } U(P_{V'^{\perp}}(x,y))=P_{V^{\perp}}(U(x,y)).$$ Thus for the induced map $\widetilde{U}\colon \Hb\to\Hb$ we obtain that
\begin{align*}
P_{\mathbb{V}}(\widetilde{U}(p))=(P_{V}(U(x,y)),0)=(U(P_{V'}(x,y)),0)=\widetilde{U}(P_{\mathbb{V}'}(p))
\end{align*}
and (recalling that the form $\sigma$ is invariant under $U$)
\begin{align*}
P_{\mathbb{V}^{\perp}}(\widetilde{U}(p))&=(P_{V^{\perp}}(U(x,y)),z-2\sigma(P_{V^{\perp}}(U(x,y)),P_{V}(U(x,y))))\\
&=(U(P_{V'^{\perp}}(x,y)),z-2\sigma(U(P_{V'^{\perp}}(x,y)),U(P_{V'}(x,y))))\\
&=(U(P_{V'^{\perp}}(x,y)),z-2\sigma(P_{V'^{\perp}}(x,y),P_{V'}(x,y)))\\
&=\widetilde{U}(P_{\mathbb{V}'^{\perp}}(p)).
\end{align*}
Thus $p\in \overline{R}_{1,V'}(0,r)$ if and only if $\widetilde{U}(p)\in \overline{R}_{1,V}(0,r)$. Since $\widetilde{U}$ preserves the metric $d_{\Hb}$, it also leaves the Hausdorff content invariant and thus $$\mc{H}_{\infty}^{t}\left( \overline{R}_{1,V'}(0,r) \right)=\mc{H}_{\infty}^{t}\left( \overline{R}_{1,V}(0,r) \right).$$
Repeating the same arguments for rectangles of type $2$ completes the proof.
\end{proof}

\begin{prop}\label{prop:upperboundrectangle}
For every $j\in \{1,2\}$ and $t\in [0,2n+2]$,
$$\mc{H}_{\infty}^{t}\left( \overline{R}_{j,V}(p,r) \right)\lesssim  \Phi^{t}_{j}(r,d),$$
where the implicit constant depends on $n$, $t$,  $d$ and $j$, but not on the radius $r$, the center $p$ or the subspace $V\in G_{h}(n,d)$ determining the horizontal and vertical subgroups.
\end{prop}

\begin{proof}
We first consider the more delicate case of rectangles of type $1$. At the end of the proof we will see how similar arguments can be used to obtain the bounds also in the simpler case of type $2$ rectangles. 

Since the metric $d_{\Hb}$ is left-invariant, the Hausdorff content $\mc{H}_{\infty}^{t}$ is also left-invariant and we may assume that $p=(0,0,0)$. 
By Lemma~\ref{lemma:hausdorffcontentofrectangleindependentofisotropicsubspace}, we may also assume that $V$ is the isotropic subspace $\R^{d}\times \{0\}\subseteq \R^{2n}$. 
Write 
$R=\overline{R}_{1,V}(0,r)$.

For points $u\in\R^{n}$ we will from now on use the shorthand notations $u^{d}=(u_1,\ldots,u_d,0,\ldots,0)$ and $u^{\perp}=(0,\ldots,0,u_{d+1},\ldots,u_n)$. Then for $p=(x,y,z)\in \Hb$ we have that
$$P_{V}(x,y)=(x^{d},0) \text{ and } P_{V^{\perp}}(x,y)=(x^{\perp},y). $$
The horizontal and vertical projections thus take the form
\begin{align*}
P_{\mathbb{V}}(p)=(P_{V}(x,y),0)=(x^{d},0,0)
\end{align*}
and
\begin{align*}
P_{\mathbb{V}^{\perp}}(p)&=\left(P_{V^{\perp}}(x,y),z-2\sigma(P_{V^{\perp}}(x,y),P_{V}(x,y))\right)\\
&=\left(x^{\perp},y,z-2\sigma((x^{\perp},y),(x^{d},0))\right)\\
&=\left(x^{\perp},y,z+2\langle x^{d},y^{d} \rangle\right).
\end{align*}
The definition of the rectangle now reads as
\begin{align*}
R&=\left\{ p\mid \Vert P_{\mathbb{V}}(p)\Vert_{\Hb}\leq r_1,\ \Vert P_{\mathbb{V}^{\perp}}(p)\Vert_{\Hb}\leq r_2 \right\}\\
&=\left\{ p\mid |x^{d}|\leq r_1,\ |(x^{\perp},y)|^{4}+|z+2\langle x^{d},y^{d}\rangle|^{2}\leq r_2^{4} \right\}.
\end{align*}
Observe that if $p\in R$, then $|(x,y)|\leq \sqrt{2}\max(r_1,r_2)$ and 
\begin{align*}
|z|\leq r_2^2+2\vert\langle x^{d},y^{d}\rangle\vert\leq r_2^2+2|x^{d}||y^{d}|\leq 3\max(r_1^{2},r_2^{2})
\end{align*}
hence
\begin{align*}
\mc{H}_{\infty}^{t}(R)\leq (\diam(R))^{t}\lesssim \max(r_1^t,r_2^t).
\end{align*}
Thus, it suffices to show that if $r_1\leq r_2$ then
$$\mc{H}_{\infty}^{t}(R)\lesssim r_1^{t+d-2(n+1)}r_2^{2(n+1)-d} $$
and if $r_1\geq r_2$ then
$$\mc{H}_{\infty}^{t}(R)\lesssim \min \left( r_1^{\frac{t+d}{2}}r_2^{\frac{t-d}{2}}, r_1^{d+1}r_2^{t-d-1}, r_1^{2n+2+d-t}r_2^{2t-(2n+2+d)}     \right) .$$

Suppose first that $r_1\leq r_2$. Then given any points $p\in R$ and $p'\in \overline{B}(p,r_1)$ we have the estimates 
$$|x'^{d}|\leq |x^{d}|+|x'^{d}-x^{d}|\leq 2r_1,\ |(x'^{\perp},y')|\leq r_1+r_2\leq 2r_2 $$
and
\begin{align*}
|z'|&\leq |z'-z-2(\langle x,y'\rangle-\langle y,x'\rangle )|+|z|+2(|x||y'|+|y||x'|)\\
&\leq r_1^2+3r_2^2+2((|x^{d}|+|x^{\perp}|)|y'|+|y|(|x'^{d}|+|x'^{\perp}|)) 
\leq 20r_2^2.
\end{align*}
Thus
$$\bigcup_{p\in R}\overline{B}(p,r_1)\subseteq \{ p'\mid |x'^{d}|\leq 2r_1, |(x'^{\perp},y)|\leq 2r_2, |z'|\leq 20r_2^2\}=:A .$$
Note that the Lebesgue measure of this set is $\lambda(A)\approx r_2^{2(n+1)-d}r_1^{d}$. 

Let $\{ p_i\}_{i=1}^{M}$ be a maximal collection of points with the properties that $p_i\in R$ for each $i$ and $d_{\Hb}(p_i,p_j)\geq r_1$ for $i\neq j$. Then the balls $B(p_i,r_1/2)$ are disjoint and 
$$R\subseteq \bigcup_{i=1}^{M}\overline{B}(p_i,r_1)\subseteq A. $$
Using the fact that $\lambda (B(p,\epsilon))=\lambda (B(0,\epsilon))\approx \epsilon^{2n+2}$ for all $p\in \Hb$ and $\epsilon>0$ we obtain that
\begin{align*}
M\lambda(B(0,r_1/2))&=\lambda \left( \bigcup_{i=1}^{M}B(p_i,r_1/2)\right)\leq \lambda (A)\approx r_2^{2(n+1)-d}r_1^{d} 
\end{align*}
which in turn yields
$$ M\lesssim \frac{r_2^{2(n+1)-d}r_1^{d}}{r_1^{2n+2}}=\left(\frac{r_2}{r_1} \right)^{2(n+1)-d}.$$
Since $R\subseteq \bigcup_{i=1}^{M}\overline{B}(p_i,r_1)$ we obtain the bound
$$\mc{H}_{\infty}^{t}(R)\lesssim \left(\frac{r_2}{r_1} \right)^{2(n+1)-d}r_1^{t}=r_1^{t+d-2(n+1)}r_2^{2(n+1)-d} $$
as desired.

We now consider the case $r_1\geq r_2$. We will first establish the bound 
\begin{equation}\label{eqn:easyupperboundfortype1}
\mc{H}_{\infty}^{t}(R)\lesssim r_1^{\frac{t+d}{2}}r_2^{\frac{t-d}{2}}. 
\end{equation}
To this end, we estimate the distance between a point $p\in R$ and its horizontal projection $(x^{d},0,0)=P_{\mathbb{V}}(p)\in \mathbb{V}$. Observe now that for $p\in R$ it holds that
\begin{align*}
|z-2\langle x^{d},y^{d}\rangle|\leq |z+2\langle x^{d},y^{d}\rangle|+4|\langle x^{d},y^{d}\rangle|\leq r_2^2+4r_1r_2\leq 5r_1r_2.
\end{align*}
Hence we obtain that
\begin{align*}
d_{\Hb}(P_{\mathbb{V}}(p),p)&=\left( |(x^{\perp},y)|^{4}+ |z-2\langle x^{d},y^{d}\rangle|^{2}\right)^{1/4}\leq \left( r_2^4+(5r_1r_2)^{2} \right)^{1/4}\\ &\leq \sqrt[4]{26}(r_1r_2)^{1/2}\leq 3(r_1r_2)^{1/2}.
\end{align*}
Now cover the Euclidean ball $\overline{B}_{\R^{d}}(0,r_1)$ with $$M\lesssim \left(r_1/\sqrt{r_1r_2}\right)^{d}=(r_1/r_2)^{d/2}$$ Euclidean balls $B_{\R^{d}}(p_i',\sqrt{r_1r_2})$. Then the points $p_i:=(p_i',0,0,0)\in \mathbb{V}.$ Given $p\in R$ there exists an index $j\in \{1,\ldots,M\}$ such that 
$$\sqrt{r_1r_2}\geq |x^{d}-(p_j',0)|=d_{\Hb}(P_{\mathbb{V}}(p),p_j) $$
hence
$d_{\Hb}(p,p_j)\leq 4(r_1r_2)^{1/2}.$ This yields the desired bound
$$\mc{H}_{\infty}^{t}(R)\lesssim (r_1/r_2)^{d/2}(r_1r_2)^{t/2}= r_1^{\frac{t+d}{2}}r_2^{\frac{t-d}{2}}. $$

To obtain the remaining upper bounds we will use similar "volume-based"\ arguments as we did in the previous case $r_1\leq r_2$. Consider a radius $\varrho\in [r_2^2/r_1,r_2]$ and suppose that $p\in R$ and $p'\in B(p,\varrho)$. 
Then $|x'^{d}|\leq r_1+\varrho\leq 2r_1$, $|(x'^{\perp},y')|\leq r_2+\varrho\leq 2r_2 $ and
\begin{align*}
|z'+2\langle x'^{d},y'^{d}\rangle|\leq\ & |z'-z-2(\langle x,y'\rangle-\langle y,x' \rangle)|+|z+2\langle x^{d},y^{d}\rangle|\\
&+2|\langle x'^{d},y'^{d}\rangle - \langle x^{d},y^{d}\rangle + \langle x,y'\rangle-\langle y,x' \rangle|\\
\leq\ & \varrho^{2}+r_2^2+2|\langle x'^{d},y'^{d}\rangle - \langle x^{d},y^{d}\rangle + \langle x^{d},y'^{d}\rangle-\langle y^{d},x'^{d} \rangle| \\
&+2|\langle x^{\perp},y'^{\perp}\rangle-\langle y^{\perp},x'^{\perp} \rangle |\\
\leq\ & \varrho^{2}+r_2^2+2|\langle x'^{d}+x^{d},y'^{d}-y^{d}\rangle|\\
&+2(|x^{\perp}||y'^{\perp}|+|y^{\perp}||x'^{\perp}|) \\
\leq\ &\varrho^{2}+r_2^2+6r_1\varrho+8r_2^2\leq 16r_1\varrho.
\end{align*}
Thus
\begin{align*}\bigcup_{p\in R}\overline{B}(p,\varrho)&\subseteq \{ p'\mid |x'^{d}|\leq 2r_1, |(x'^{\perp},y')|\leq 2r_2,|z'+2\langle x'^{d},y'^{d}\rangle| \leq 16r_1\varrho \} \\
&=:A_{\varrho}.
\end{align*}
Note that $\lambda(A_{\varrho})\approx r_1^{d+1}r_2^{2n-d}\varrho$. The argument used above in the case $r_1\leq r_2$ implies that we can cover $R$ with $M$ balls of radius $r_2$, or with $N$ balls of radius $r_2^2/r_1$, where
\begin{align*}
M&\lesssim \frac{r_1^{d+1}r_2^{2n-d}r_2}{r_2^{2n+2}}=\left(\frac{r_1}{r_2}\right)^{d+1}
\end{align*}
and
\begin{align*}
N&\lesssim \frac{r_1^{d+1}r_2^{2n-d}r_2^2/r_1}{(r_2^2/r_1)^{2n+2}}=\left( \frac{r_1}{r_2}\right)^{2n+2+d}.
\end{align*}
Thus
\begin{align*}
\mc{H}_{\infty}^{t}(R)&\lesssim \min\left( \left(\frac{r_1}{r_2}\right)^{d+1}r_2^t, \left( \frac{r_1}{r_2}\right)^{2n+2+d} r_2^{2t}r_1^{-t}   \right)\\
&=\min \left( r_1^{d+1}r_2^{t-d-1}, r_1^{2n+2+d-t}r_2^{2t-(2n+2+d)} \right)
\end{align*}
and the proof of the upper bound for rectangles of type $1$ is complete. 

We now consider the case of rectangles of type $2$. Proceeding as we did in the previous case we may once again assume that the center $p$ is the origin and that $V=\R^{d}\times \{0\}$. The rectangle then takes the form
\begin{align*}
R':=\overline{R}_{2,V}(0,r)&=\left\{ p\mid |x^{d}|\leq r_1,\ |(x^{\perp},y)|^{4}+|z-2\langle x^{d},y^{d}\rangle|^{2}\leq r_2^{4} \right\}.
\end{align*}
Note that by the argument given in the previous case we still have the same trivial bounds as before, that is, 
$\mc{H}_{\infty}^{t}(R')\leq |R'|^{t}\lesssim \max(r_1^t,r_2^t).$ In the case $r_1\leq r_2$ we proceed exactly as we did for the type $1$ rectangles and we find that 
$$\bigcup_{p\in R'}\overline{B}(p,r_1)\subseteq A,$$
where $A$ is the set defined in the type $1$ situation. Thus, we obtain the same upper bound as for the type $1$ rectangles and the case $r_1\leq r_2$ is proved. In the case $r_1\geq r_2$ we follow similar arguments as we did to establish the upper bound (\ref{eqn:easyupperboundfortype1}). For a point $p\in R'$ we have from the definition of $R'$ that 
\begin{align*}
d_{\Hb}(P_{\mathbb{V}}(p),p)&=\left( |(x^{\perp},y)|^{4}+ |z-2\langle x^{d},y^{d}\rangle|^{2}\right)^{1/4}\leq r_2.
\end{align*}
Thus by covering the Euclidean ball $\overline{B}_{\R^{d}}(0,r_1)$ with $M'\lesssim \left(r_1/r_2\right)^{d}$ Euclidean balls $\overline{B}_{\R^{d}}(p_i^{d},r_2)$ we see as before that for any $p\in R'$ there exists an index $i\in \{1,\ldots, M'\}$ such that $d_{\Hb}(p,p_i)\leq 2r_2$, where $p_i$ is defined in the same way as in the proof of (\ref{eqn:easyupperboundfortype1}). This yields the bound 
\begin{align*}
\mc{H}_{\infty}^{t}(R')\lesssim r_1^{d}r_2^{t-d}
\end{align*}
and the proposition is completely proved.
\end{proof}

\begin{prop}\label{prop:lowerboundrectangle}
For every isotropic subspace $V\in G_{h}(n,d)$, a pair of positive numbers $r=(r_1,r_2)$ and for every $t\in ]0,2n+2[\setminus \Z$ we have the following:

If $r_1\geq r_2$, $d=1$ and $t\in ]3,2n+1[\setminus \Z$, then for every $s\in ]\lfloor t\rfloor,t[$ we have that
\begin{align}\label{eqn:lowerboundincased=1}
\capa_{s}(\overline{R}_{1,V}(0,r))\gtrsim \Phi^{t}_{1}(r,d)r_1^{s-t}.
\end{align} 
For all other possible choices of $V\in G_{h}(n,d)$, $j\in \{1,2\}$, $r$ and $t$ we have that
\begin{align}\label{eqn:lowerboundinmostcases}
\capa_{t}(\overline{R}_{j,V}(0,r))\gtrsim \Phi^{t}_{j}(r,d).
\end{align} 
The implicit constants above depend on $n$, $t$, $s$, $j$ and $d$, but not on $r$ or $V\in G_{h}(n,d)$.
\end{prop}

\begin{proof}
As in the proof of Proposition~\ref{prop:upperboundrectangle} we begin by considering the more difficult case of type $1$ rectangles. We will obtain the desired results for the rectangles of type $2$ at the end of the proof.

Let $R=\overline{R}_{1,V}(0,r)$. Directly from the definition of capacity we obtain that
$$\capa_t(R)\geq \frac{\lambda(R)^2}{I_t(R)}.$$
Recall that $d_{\Hb}$ and $\lambda$ are invariant under the map $\widetilde{U}\colon \Hb\to\Hb$ induced by $U\in U(n)$. Also, we recall from the proof of Proposition~(\ref{prop:upperboundrectangle}) that $U(n)$ acts transitively on $G_{h}(n,d)$ and if $U\in U(n)$ maps $V'$ to $V$, then the induced map $\widetilde{U}$ maps $\overline{R}_{1,V'}(0,r)$ to $\overline{R}_{1,V}(0,r)$. Thus, to estimate the quantity $\frac{\lambda(R)^2}{I_t(R)}$, we may again assume that $V=\R^{d}\times \{0\}$. The rectangle is then again the set
\begin{align*}
R&=\left\{ p\mid |x^{d}|\leq r_1,\ |(x^{\perp},y)|^{4}+|z+2\langle x^{d},y^{d}\rangle|^{2}\leq r_2^{4} \right\}.
\end{align*}
Observe that for any fixed $x^d$ with $|x^{d}|\leq r_1$, the slice 
$$\{ p'\in R\mid x'^{d}=x^{d}\} $$
has Lebesgue measure comparable to $r_2^{2n+2-d}$. Hence, Fubini's theorem implies that $\lambda(R)\approx r_1^{d}r_2^{2n+2-d}$. 
We thus obtain the bound
\begin{equation}\label{eqn:capacityatleastsomethingoverenergy}
\capa_t(R)\gtrsim \frac{r_1^{2d}r_2^{4n+4-2d}}{I_t(R)}.
\end{equation}

In the following, we will show that if $r_1\leq r_2$, then 
\begin{align*}
I_t(R)\lesssim \begin{cases}
r_1^{2d}r_2^{4n+4-2d-t}& \text{ if } t\in ]0,2n+2-d[\setminus \Z,\\
r_1^{2n+2+d-t}r_2^{2n+2-d}& \text{ if } t\in ]2n+2-d,2n+2[\setminus \Z,
\end{cases}
\end{align*}
and if $r_1\geq r_2$, then 
\begin{align*}
I_t(R)\lesssim \begin{cases}
r_1^{2d-t}r_2^{4n+4-2d}& \text{ if } t\in ]0,d[\setminus\Z,\\
r_1^{\frac{3d-t}{2}}r_2^{\frac{8n+8-3d-t}{2}}& \text{ if } t\in ]d,d+2[\setminus\Z,\\
r_1^{d-1}r_2^{4n+5-d-t}& \text{ if } t\in ]d+2,2n+1[\setminus\Z \text{ and } d\geq 2,\\
r_2^{4n+4-t}\log\left( \frac{2r_1}{r_2}\right)& \text{ if } t\in ]3,2n+1[\setminus\Z\text{ and } d=1,\\
r_1^{t+d-2n-2}r_2^{6n+6-d-2t}& \text{ if } t\in ]2n+1,2n+2[.
\end{cases}
\end{align*}
Apart from the bound involving the logarithm, the rest of the above bounds combined with (\ref{eqn:capacityatleastsomethingoverenergy}) clearly imply (\ref{eqn:lowerboundinmostcases}). We show that (\ref{eqn:lowerboundincased=1}) follows from the logarithm bound at the end of this proof. 

Note that $d_{\Hb}$ and $\lambda$ are invariant under the map $\widetilde{U}\colon \Hb\to\Hb$ induced by $U\in U(n)$. Recall from the proof of Proposition~(\ref{prop:upperboundrectangle}) that $U(n)$ acts transitively on $G_{h}(n,d)$ and if $U\in U(n)$ maps $V'$ to $V$, then the induced map $\widetilde{U}$ maps $\overline{R}_{1,V'}(0,r)$ to $\overline{R}_{1,V}(0,r)$. Thus, to estimate $I_t(R)$, we may again assume that $V=\R^{d}\times \{0\}$. 
Let $$R_t(p)=\int_{R}d_{\Hb}(p,q)^{-t}\, \lambda(q)$$
so that 
$I_t(R)=\int_{R}R_t(p)\, d\lambda(p)$.
Let $p=(\rho,y_0,z_0)\in R$ and for $q=(x,y,z)$ let
\begin{align*}
&g_p(q)\\
&=\max\left( |x^{d}-\rho^{d}|, |(x^{\perp},y)-(\rho^{\perp},y_0)|,|z-z_0-2(\langle \rho, y \rangle-\langle y_0,x \rangle ) |^{1/2} \right). 
\end{align*}
Note that
$$R\subseteq \{ q\mid |x^d|\leq r_1, |(x^{\perp},y)|\leq r_2, |z+2\langle x^d,y^d\rangle|\leq r_2^2\}=:A. $$
For $a\geq 0$ define $B(a):=\{ q\mid g_p(q)\leq a\}$. Since $g_p(q)\approx d_{\Hb}(p,q)$, we obtain that
\begin{align*}
R_t(p)&\lesssim \int_{A} g_p(q)^{-t}\, d\lambda(q)=\int_{0}^{\infty}\lambda(A\cap \{ q\mid g_p(q)^{-t}\geq a\})\, da\\
&\approx \int_{0}^{\infty}\lambda(A\cap B(a))a^{-(t+1)}da.
\end{align*}
We must now estimate the Lebesgue measure of the set 
\begin{align*}
&A\cap B(a)\\
&=\left\{ q\, \bigg| \begin{array}{cl}
&|x^d|\leq r_1, |(x^\perp,y)|\leq r_2, |z+2\langle x^d,y^d\rangle|\leq r_2^2, |x^d-\rho^d|\leq a,\\ &|(x^\perp,y)-(\rho^\perp,y_0)|\leq a,|z-z_0-2(\langle \rho, y\rangle-\langle y_0,x\rangle) |\leq a^2 
\end{array} \right\}.
\end{align*}

Suppose first that $r_1\leq r_2$ and let $t\in ]0,2n+2[\setminus \Z$. 
Note that for any fixed $(x,y)$, the $z$-coordinate of a point $q\in A\cap B(a)$ must come from an interval of length $\min(2r_2^2,2a^2)$. Taking into account the constraints on the first $2n$-coordinates, Fubini's theorem yields
\begin{align*}
\lambda(A\cap B(a))&\lesssim \min\left( a^2 a^{2n-d}a^{d},a^{2}a^{2n-d}r_1^{d},r_2^2r_2^{2n-d}r_1^{d} \right)\\
&=\min \left( a^{2n+2}, r_1^{d}a^{2n+2-d},r_1^{d}r_2^{2n+2-d} \right).
\end{align*}
We thus obtain
\begin{align*}
R_t(p)&\lesssim \int_{0}^{r_1}a^{2n+2}a^{-(t+1)}\, da+\int_{r_1}^{r_2}r_1^{d}a^{2n+2-d}a^{-(t+1)}\, da\\
&\quad +\int_{r_2}^{\infty}r_1^{d}r_2^{2n+2-d}a^{-(t+1)}\, da\\
&\lesssim \max\left( r_1^{2n+2-t}, r_1^dr_2^{2n+2-d-t} \right).
\end{align*}
Hence
\begin{align*}
I_t(R)&\lesssim \lambda(R)\max\left( r_1^{2n+2-t}, r_1^dr_2^{2n+2-d-t} \right)\\
&\approx \max\left( r_1^{2n+2+d-t}r_2^{2n+2-d},r_1^{2d}r_2^{4n+4-2d-t}\right)\\
&=\begin{cases}
r_1^{2d}r_2^{4n+4-2d-t} &\text{ if } t\leq 2n+2-d,\\
r_1^{2n+2+d-t}r_2^{2n+2-d} &\text{ if } t\geq 2n+2-d.
\end{cases}
\end{align*}
The case $r_1\leq r_2$ is thus proved. 

Suppose then that $r_1\geq r_2$ and let $t\in ]0,2n+2[$. In this case we need a more careful estimate for the Lebesgue measure of 
$A\cap B(a).$ We collect the required upper bounds in the following sublemma.

\begin{sublemma}\label{sublemma1}
For every $a\geq 0$ we have that 
$$\lambda(A\cap B(a))\lesssim \min \left( a^{2n+2}, a^{d}r_2^{2n+2-d}, r_1^{d}r_2^{2n+2-d}\right). $$
If $0<a\leq |\rho^{d}|$, then we also have the bound
$$\lambda(A\cap B(a))\lesssim \min\left( \frac{a^{2n+1}r_2^2}{|\rho^{d}|}, \frac{a^{d+2}r_2^{2n+1-d}}{|\rho^{d}|} \right). $$
\end{sublemma}

\begin{proof}
In the proof of this lemma, to make reading less cumbersome, we deviate from our usual notation slightly and we write $x^{d}=(x_1,\ldots,x_d)\in \R^{d}$ instead of $x^d=(x_1,\ldots,x_d,0,\ldots,0)\in \R^n$ and similarly we write
$x^{\perp}=(x_{d+1},\ldots ,x_n)\in \R^{n-d}$.

For every $a\geq 0$, we note as before that for any fixed first $2n$ coordinates $(x,y)$, the $z$-coordinate of a point $q\in A\cap B(a)$ must belong to the intersection of two intervals of lengths $2r_2^2$ and $2a^2$, respectively. Noting also that then
$$x^{d}\in \overline{B}_{\R^d}(0,r_1)\cap \overline{B}_{\R^d}(\rho^{d},a)  $$
and
$$(x^{\perp},y)\in \overline{B}_{\R^{2n-d}}(0,r_2)\cap \overline{B}_{\R^{2n-d}}((\rho^{\perp},y_0),a),$$
Fubini's theorem yields
\begin{align*}
\lambda(A\cap B(a))&\lesssim \min\left( a^{2n+2}, a^{d}r_2^{2n+2-d},r_1^{d}r_2^{2n+2-d} \right).
\end{align*}
Suppose then that $a\leq |\rho^{d}|$. Assume first that $\rho^{d}=(|\rho^{d}|,0,\ldots,0)\in \R^d$ so that $\langle \rho , y\rangle= |\rho^{d}|y_1$ and write 
$$\widehat{x}=(x_2,\ldots, x_d)\in \R^{d-1}.$$ Observe that if $q=(x,y,z)\in A\cap B(a)$, then
\begin{align}\label{eqn:betterboundforrho1}
-a^2 -z_0-2(\langle \rho, y \rangle-\langle y_0,x \rangle)\leq -z \leq a^2 - z_0-2(\langle \rho, y \rangle-\langle y_0,x \rangle)
\end{align}
and, on the other hand,
\begin{align}\label{eqn:betterboundforrho2}
-r_2^2-z\leq 2\langle x^{d},y^{d}\rangle\leq r_2^2-z .
\end{align}
Substituting (\ref{eqn:betterboundforrho1}) into (\ref{eqn:betterboundforrho2}) and subtracting the term $2\langle \widehat{x},\widehat{y} \rangle$, we obtain
\begin{align*}
-&r_2^2-a^2 -z_0-2(\langle \rho, y \rangle-\langle y_0,x \rangle + \langle \widehat{x},\widehat{y} \rangle)\\ 
\leq &2x_1y_1\\ 
\leq &r_2^2+a^2 - z_0-2(\langle \rho, y \rangle-\langle y_0,x \rangle+ \langle \widehat{x},\widehat{y} \rangle).
\end{align*}
Write $$C=C(x,\widehat{y},y^{\perp}):=-z_0-2(\langle \rho^{\perp}, y^{\perp} \rangle-\langle y_0,x \rangle + \langle \widehat{x},\widehat{y} \rangle) .$$
Noting that $\langle \rho,y \rangle=|\rho^{d}|y_1+\langle \rho^{\perp},y^{\perp}\rangle$ yields
\begin{align}\label{eqn:betterboundforrho3}
C-r_2^2-a^2 \leq 2(x_1+|\rho^{d}|)y_1 \leq C+r_2^2+a^2.
\end{align}
Since $q\in A\cap B(a)$ we have that $|x_1-|\rho^{d}||\leq |x^{d}-\rho^{d}|\leq a$, and by our assumption $a\leq |\rho^{d}|$ we obtain
\begin{align*}
0<|\rho^{d}|\leq 2|\rho^{d}|-a\leq x_1+|\rho^{d}|.
\end{align*}
From (\ref{eqn:betterboundforrho3}) we now obtain
\begin{align*}
\frac{C-r_2^2-a^2}{2(x_1+|\rho^{d}|)}\leq y_1\leq \frac{C+r_2^2+a^2}{2(x_1+|\rho^{d}|)}.
\end{align*}
Thus for every point in $A\cap B(a)$ with the coordinates $x$, $\widehat{y}$ and $y^{\perp}$ fixed, the $y_1$ coordinate must belong to the same interval of length 
$$\frac{r_2^2+a^2}{x_1+|\rho^{d}|}\leq \frac{r_2^2+a^2}{|\rho^{d}|}.$$ By Fubini's theorem
\begin{align*}
&\lambda(A\cap B(a))=\\ 
&\int \int \int \int \chi_{A\cap B(a)}(x,y,z)\, d\mc{L}^{1}(z)\, d\mc{L}^{1}(y_1)\, d\mc{L}^{d}(x^{d})\, d\mc{L}^{2n-d-1}(x^{\perp},\widehat{y},y^{\perp})\\
&\lesssim \min\left( a^2,r_2^2 \right)\frac{r_2^2+a^2}{|\rho^{d}|}\min \left( a^d,r_1^d \right)\min\left( a^{2n-d-1},r_2^{2n-d-1} \right)\\
&\lesssim \frac{r_2^2a^2}{|\rho^{d}|}\min \left( a^d,r_1^d \right)\min\left( a^{2n-d-1},r_2^{2n-d-1}\right)\\
&\leq \min \left(\frac{r_2^2a^{2n+1}}{|\rho^{d}|}, \frac{a^{d+2}r_2^{2n+1-d}}{|\rho^{d}|} \right).
\end{align*}
This completes the proof in the case $\rho^{d}=(|\rho^{d}|,0,\ldots ,0)$. It remains to see that this is indeed  sufficient. 

Given an arbitrary $\rho^{d}\in \R^{d}\setminus \{0\}$ there exists an orthogonal map $O\in O(d)$ with the property that $O(\rho^{d})=(|\rho^{d}|,0\ldots,0)$.  This induces a map $\widetilde{O}\in O(2n+1)$ via the formula
$$\widetilde{O}(x,y,z)=(O(x^d),x^{\perp},O(y^d),y^{\perp},z). $$
As the map $\widetilde{O}$ preserves the norms and inner products appearing in the definition of $A\cap B(a)$ we have that
$$A\cap B(a)=\widetilde{O}^{-1}\left(A\cap \{ q\mid g_{\widetilde{O}(p)}(q)\leq a\} \right). $$
Since $\lambda$ is invariant under $\widetilde{O}$ it follows that
$$\lambda(A\cap B(a))=\lambda(A\cap \{ q\mid g_{\widetilde{O}(p)}(q)\leq a\} ).$$
This completes the proof since the first coordinate of $\widetilde{O}(p)$ is $|\rho^{d}|$ and the following $(d-1)$ coordinates are all equal to zero. 
\end{proof}
We now continue with the proof of the proposition. By Sublemma~\ref{sublemma1}, it is true for every $p\in R$ that
\begin{align}\label{eqn:upperboundrtpforallrho}
R_t(p)\lesssim& \int_{0}^{r_2}a^{2n+2-(t+1)}\, da+ \int_{r_2}^{r_1}r_2^{2n+2-d} a^{d-(t+1)}\,da\nonumber\\
&+\int_{r_1}^{\infty}r_1^{d}r_2^{2n+2-d}a^{-(t+1)}\, da \nonumber\\
\lesssim& \max\left( r_2^{2n+2-t}, r_2^{2n+2-d}r_1^{d-t} \right)\nonumber\\
=&\begin{cases}
r_2^{2n+2-d}r_1^{d-t}& \text{ if } t<d,\\
r_2^{2n+2-t}& \text{ if } t>d.
\end{cases}
\end{align}
For $0<t<d$ we thus obtain
\begin{align*}
I_t(R)&\lesssim \lambda(R)r_2^{2n+2-d}r_1^{d-t}\approx r_1^{d}r_2^{2n+2-d}r_2^{2n+2-d}r_1^{d-t}=r_1^{2d-t}r_2^{4n+4-2d}
\end{align*}
as desired. Suppose then that $t>d$. Observe that if $|\rho^{d}|\geq r_2$, then $$r_2^2/|\rho^{d}|\leq r_2\leq \sqrt{r_2|\rho^{d}|}$$ and for $a\in \left[r_2^2/|\rho^{d}|,\sqrt{r_2|\rho^{d}|}\right]$ it is true that $a\leq |\rho^{d}|$. For such $|\rho^{d}|$, Sublemma~\ref{sublemma1} yields that
\begin{align}\label{eqn:upperboundrtpforlargerho}
R_t(p)\lesssim &\int_{0}^{r_2^2/|\rho^{d}|}a^{2n+2-(t+1)}\, da +\int_{r_2^2/|\rho^{d}|}^{r_2}\frac{r_2^2a^{2n+1-(t+1)}}{|\rho^{d}|}\,da\nonumber\\
&+ \int_{r_2}^{\sqrt{r_2|\rho^{d}|}}\frac{a^{d+2-(t+1)}r_2^{2n+1-d}}{|\rho^{d}|}\, da+\int_{\sqrt{r_2|\rho^{d}|}}^{r_1} a^{d-(t+1)}r_2^{2n+2-d}\, da\nonumber\\
&+ \int_{r_1}^{\infty}r_1^{d}r_2^{2n+2-d}a^{-(t+1)}\, da\nonumber\\
\lesssim &\max\left( \left(\frac{r_2^2}{|\rho^{d}|} \right)^{2n+2-t}, \frac{r_2^{2n+3-t}}{|\rho^{d}|}, r_2^{2n+2-d}(r_2|\rho^{d}|)^{\frac{d-t}{2}}, r_2^{2n+2-d}r_1^{d-t} \right)\nonumber\\
\mathrel{\overset{\ast}{=}} &\begin{cases}
r_2^{2n+2-d}(r_2|\rho^{d}|)^{\frac{d-t}{2}}& \text{ if } d<t<d+2,\\
\frac{r_2^{2n+3-t}}{|\rho^{d}|}& \text{ if } d+2<t<2n+1,\\
\left(\frac{r_2^2}{|\rho^{d}|} \right)^{2n+2-t}& \text{ if } 2n+1<t<2n+2.
\end{cases}
\end{align}
To see that the equality $\ast$ holds, denote the options in the maximum on the second-to-last line by $O_1,O_2,O_3$ and $O_4$, respectively. Note first that $O_4\geq O_3$ is equivalent to $t\leq d$ so the last term is never the largest one for $t>d$. To compare the rest of the terms we calculate
\begin{align*}
\frac{O_3^2}{O_2^2}=\frac{r_2^{4n+4-2d+d-t}|\rho^{d}|^{d-t}}{r_2^{4n+6-2t}|\rho^{d}|^{-2}}=\left( \frac{r_2}{|\rho^{d}|} \right)^{t-d-2}\leq 1 \Leftrightarrow t\geq d+2.
\end{align*}
Thus $O_2\geq O_3$ is equivalent to $t\geq d+2$. We also have that
\begin{align*}
\frac{O_1}{O_2}&=\frac{r_2^{2(2n+2-t)}|\rho^{d}|^{t-2n-2}}{r_2^{2n+3-t}|\rho^{d}|^{-1}}=\left(\frac{r_2}{|\rho^{d}|} \right)^{2n+1-t}\leq 1 \Leftrightarrow t\leq 2n+1
\end{align*}
hence the inequality $O_2\geq O_1$ is equivalent to $t\leq 2n+1$. Combining these estimates we obtain the last equality in (\ref{eqn:upperboundrtpforlargerho}). 

For $\tau\in [0,r_1]$, let $h_t(\tau)=\sup_{|\rho^{d}|=\tau}R_t(p)$ so that 
$$I_t(R)\lesssim r_2^{2n+2-d}\int_{0}^{r_1}h_t(\tau)\tau^{d-1}\, d\tau. $$
If $d<t<d+2$, we obtain from (\ref{eqn:upperboundrtpforallrho}) and (\ref{eqn:upperboundrtpforlargerho}) that
\begin{align*}
I_t(R)&\lesssim r_2^{2n+2-d}\left( \int_{0}^{r_2}r_2^{2n+2-t}\tau^{d-1}\, d\tau+\int_{0}^{r_1}r_2^{\frac{4n+4-d-t}{2}}\tau^{\frac{3d-t}{2}-1}\, d\tau \right)\\
&\lesssim \max \left(r_2^{4n+4-t},r_2^{\frac{8n+8-3d-t}{2}}r_1^{\frac{3d-t}{2}} \right)=r_2^{\frac{8n+8-3d-t}{2}}r_1^{\frac{3d-t}{2}},
\end{align*}
where the last equality follows from
\begin{align*}
r_2^{\frac{8n+8-3d-t}{2}}r_1^{\frac{3d-t}{2}}=r_2^{4n+4-t}\left( \frac{r_1}{r_2}\right)^{\frac{3d-t}{2}}
\end{align*}
and the facts that $r_1\geq r_2$ and $t<d+2\leq 3d$. If $2n+1<t<2n+2$, the estimates (\ref{eqn:upperboundrtpforallrho}) and (\ref{eqn:upperboundrtpforlargerho}) yield
\begin{align*}
I_t(R)&\lesssim r_2^{2n+2-d}\int_{0}^{r_2}r_2^{2n+2-t}\tau^{d-1}\, d\tau+ r_2^{2n+2-d}\int_{0}^{r_1}\frac{r_2^{2(2n+2-t)}}{\tau^{2n+2-t}}\tau^{d-1}\,d\tau\\
&\lesssim r_2^{4n+4-t}+r_2^{6n+6-d-2t}r_1^{t+d-2n-2}\\
&\lesssim r_2^{6n+6-d-2t}r_1^{t+d-2n-2},
\end{align*}
where the last estimate can be seen by calculating
\begin{align*}
\frac{r_2^{6n+6-d-2t}r_1^{t+d-2n-2}}{r_2^{4n+4-t}}=\left(\frac{r_2}{r_1} \right)^{2n+2-d-t}\geq 1,
\end{align*}
since $r_2\leq r_1$ and $t>2n+1\geq 2n+2-d$. 

If $d+2<t<2n+1$ and $d\geq 2$, we obtain that
\begin{align*}
I_t(R)&\lesssim r_2^{4n+4-t}+ r_2^{2n+2-d}\int_{0}^{r_1}\frac{r_2^{2n+3-t}}{\tau}\tau^{d-1}\, d\tau \\
&\lesssim \max\left(r_2^{4n+4-t}, r_2^{4n+5-d-t}r_1^{d-1} \right)= r_2^{4n+5-d-t}r_1^{d-1}.
\end{align*}
Combining all the upper bounds for $I_t(R)$ obtained thus far, we have completely proved the estimate (\ref{eqn:lowerboundinmostcases}). 

We now consider the case $d=1$ and $t\in ]3,2n+1[\setminus \Z$. Once again, from (\ref{eqn:upperboundrtpforallrho}) and (\ref{eqn:upperboundrtpforlargerho}) we obtain
\begin{align*}
I_t(R)&\lesssim r_2^{4n+4-t}+ r_2^{2n+2-1}\int_{r_2}^{r_1}\frac{r_2^{2n+3-t}}{\tau}\, d\tau \\
&=r_2^{4n+4-t}+ r_2^{4n+4-t}\log\left( \frac{r_1}{r_2}\right)\\
&\lesssim r_2^{4n+4-t}\log \left( \frac{2r_1}{r_2}\right).
\end{align*}
This yields
\begin{align*}
\capa_{t}(R)\gtrsim \frac{\Phi^t_{1}(r,1)}{\log \left( \frac{2r_1}{r_2}\right)}.
\end{align*}
Thus if $s\in ]\lfloor t \rfloor,t[$, we have that
\begin{align*}
\capa_{s}(R)\gtrsim \frac{\Phi^s_{1}(r,1)}{\log \left( \frac{2r_1}{r_2}\right)}=\frac{r_1^2r_2^{s-2}}{\log \left( \frac{2r_1}{r_2}\right)}=\Phi^t_{1}(r,1)\frac{r_2^{s-t}}{\log \left( \frac{2r_1}{r_2}\right)}.
\end{align*}
Consider the function $f\colon ]0,r_1[\to \R$, $f(x)=x^{t-s}\log(2r_1/x)$. Note that
\begin{align*}
f'(x)
&=x^{t-s-1}\left( (t-s)\log(2r_1/x)-1 \right)
\end{align*}
and the zero of the derivative is $$2e^{1/(s-t)}r_1=:c_0r_1.$$ Since $t-s-1<0$, the derivative $f'$ is strictly decreasing and thus the function $f$ attains its maximum value at $c_0r_1$ and we obtain the bound
\begin{align*}
\capa_{s}(R)\gtrsim \frac{\Phi^t_{1}(r,1)}{(c_0r_1)^{t-s}\log\left( \frac{2r_1}{c_0r_1}\right)}=\frac{\Phi^t_{1}(r,1)}{(c_0r_1)^{t-s}}(t-s)\gtrsim \Phi^t_{1}(r,1)r_1^{s-t}.
\end{align*}
This completes the proof of (\ref{eqn:lowerboundincased=1}) and thus also the proof of the proposition for rectangles of type $1$. 

It remains to obtain the bound (\ref{eqn:lowerboundinmostcases}) for type $2$ rectangles.  Write $R'=\overline{R}_{2,V}(0,r)$. 
It suffices to show that if $r_1\leq r_2$, the same bound as in the type $1$ case holds. If $r_1\geq r_2$ it suffices to show that
\begin{align*}
I_t(R')\lesssim \begin{cases}
r_1^{2d-t}r_2^{4n+4-2d}& \text{ if } t\in ]0,d[\setminus\Z,\\
r_1^{d}r_2^{4n+4-d-t}& \text{ if } t\in ]d,2n+2[\setminus\Z,\\
\end{cases}
\end{align*}

 We may once again assume that $V=\R^{d}\times \{0\}$. We proceed exactly as in the previous case of type $1$ rectangles with the only difference being that we replace the condition $|z+2\langle x^d,y^d\rangle|$ with $|z-2\langle x^d,y^d\rangle| $ in the definition of the set $A$. Following the arguments used in the type $1$ case we immediately obtain the desired bound in the case $r_1\leq r_2$, since this bound is the same as before. 

In the case $r_1\geq r_2$, to estimate the measure $\lambda(A\cap B(a))$ we only need the trivial bound 
$$\lambda(A\cap B(a))\lesssim \min\left(a^{2n+2}, a^{d}r_2^{2n+2-d},r_1^{d}r_2^{2n+2-d} \right) $$ 
which is obtained in exactly the same way as in the proof of the type $1$ bounds. This yields
\begin{align*}
R_t(p)\lesssim& \int_{0}^{r_2}a^{2n+2-(t+1)}  \, da+\int_{r_2}^{r_1} a^{d-(t+1)}r_2^{2n+2-d}\, da\\
&+\int_{r_1}^{\infty}r_1^{d}r_2^{2n+2-d}a^{-(t+1)} \, da\\
\lesssim& \max\left( r_2^{2n+2-t},r_2^{2n+2-d}r_1^{d-t} \right)\\
=&
\begin{cases}
r_2^{2n+2-d}r_1^{d-t}& \text{ if } t\in ]0,d[\setminus \Z,\\
r_2^{2n+2-t}& \text{ if } t\in ]d,2n+2[\setminus \Z. 
\end{cases}
\end{align*}
From this we obtain that
\begin{align*}
I_t(R')&\lesssim \lambda(R')\sup_{p\in R'}R_t(p)\lesssim \begin{cases}
r_1^{2d-t}r_2^{4n+4-2d}& \text{ if } t\in ]0,d[\setminus \Z,\\
r_1^{d}r_2^{4n+4-d-t}& \text{ if } t\in ]d,2n+2[\setminus \Z 
\end{cases}
\end{align*}
and the proof of the proposition is complete.
\end{proof}

\section{Proof of Theorem \ref{theorem1.1}}\label{section:proofofmainresult}

To complete the proof of Theorem \ref{theorem1.1} we make use of the following theorem. For a proof of this, see \cite[Theorem~1.3]{ekstrometal}.

\begin{theorem}\label{theorem1.3}
Let $G$ be a unimodular group with left-invariant metric and $(c,d)$-regular Haar measure $\lambda$, and let $W\subseteq G$ be a bounded open set such that $\lambda(W)=1$. Define the probability space $(\Omega,\mathbb{P})$ by $\Omega=G^{\N}$ and $\mathbb{P}=(\lambda_{|_{W}})^{\N}$. Let $(V_k)$ be a bounded sequence of open subsets of $G$. For $\omega=(\omega_k)\in \Omega$, let 
$$ E(\omega)=\limsup_{k\to \infty}(\omega_kV_k).$$
If $t\in ]0,d[$ and 
$$ \sum_{k}\capa_t(V_k)=\infty$$
then there exists $g\in G$ so that almost every $\omega$ is such that $\mc{H}^{t}(\mc{U}\cap E(\omega))=\infty$ for every open subset $\mc{U}$ of $Wg$.
\end{theorem}

\begin{proof}[Proof of Theorem \ref{theorem1.1}]
Clearly $\dim_{\mathrm{H}}E_{\mathbf{j},\underline{V},\underline{r}}(\omega)\leq 2n+2$ for every $\omega$. Suppose that $t\in [0,2n+2]$ is such that $\sum_{k}\Phi^{t}_{j_k}(r_k,d_k)<\infty$. Then, for any $\omega\in \Omega$ and $k_0\in \N$, we have by Proposition~\ref{prop:upperboundrectangle} that
\begin{align*}
\mc{H}_{\infty}^{t}(E_{\mathbf{j},\underline{V},\underline{r}}(\omega))\leq \sum_{k=k_0}^{\infty}\mc{H}_{\infty}^{t}(\overline{R}_{j_k,V_k}(\omega_k,r_k) )\lesssim \sum_{k=k_0}^{\infty} \Phi^{t}_{j_k}(r_k,d_k).
\end{align*}
Thus $\mc{H}_{\infty}^{t}(E_{\mathbf{j},\underline{V},\underline{r}}(\omega))=0$ and $\dim_{\mathrm{H}}E_{\mathbf{j},\underline{r}}(\omega)\leq t$. This implies that 
 $$\dim_{\mathrm{H}}E_{\mathbf{j},\underline{V},\underline{r}}(\omega)\leq \inf\{ t\mid \sum_{k}\Phi^{t}_{j_k}(r_k,d_k)<\infty\}$$ for every $\omega$.

 Suppose then that $\sum_{k}\Phi^{t}_{j_k}(r_k,d_k)=\infty$ for some $t\in [0,2n+2]\setminus \Z$. Since the sequence $\underline{r}$ is bounded, Proposition \ref{prop:lowerboundrectangle} implies that for $s\in ]\lfloor t\rfloor, t[$, the $s$-capacities of (interiors of) the rectangles $\overline{R}_{j_k,V_k}(0,r_k)$ sum up to infinity and thus Theorem \ref{theorem1.3} (with $\lambda(W)^{-1}\lambda$ in place of $\lambda$) yields $\dim_{\mathrm{H}}E_{\mathbf{j},\underline{V},\underline{r}}(\omega) \geq s$ for $\mathbb{P}$-almost every $\omega$. By letting $s\nearrow t$ through some countable sequence we obtain that $\dim_{\mathrm{H}}E_{\mathbf{j},\underline{V},\underline{r}}(\omega) \geq t$ for $\mathbb{P}$-almost every $\omega$. Finally, letting
 $$t\nearrow\inf\{ t'\mid \sum_{k}\Phi^{t'}_{j_k}(r_k,d_k)<\infty\}\wedge(2n+2)$$ through some sequence completes the proof.
\end{proof}

\begin{remark}\label{remark:otherrectangles}
In \cite{ekstrometal}, the closed rectangle of radius $r=(r_1,r_2)$ centered at the origin is defined to be the set
\begin{align*}
\overline{R}(0,r):= \{ p'\in \Hb\mid \Vert P_{H(0)}(p')\Vert_{\Hb}\leq r_1, \Vert P_{L(0)}(p')\Vert_{\Hb}\leq r_2\},
\end{align*}
where $P_{L(0)}$ and $P_{H(0)}$  are the Euclidean projections (recalling that as a set, $\Hb^{n}=\R^{2n+1}$) onto the vertical line and the horizontal plane through the origin, respectively. The corresponding rectangle centered at $p$ is defined analogously to our definition, that is, $\overline{R}(p,r):=p\overline{R}(0,r).$
In \cite{ekstrometal}, rectangles are defined only in the first Heisenberg group $\Hb^{1}$, but the definition given above works equally well in all Heisenberg groups. 

Define the probability space $(\Omega,\mathbb{P})$ in the same way as before. Given a sequence $\underline{r}:=(r_k)$ of pairs of positive numbers and $\omega\in \Omega$, let $$E_{\underline{r}}(\omega)=\limsup_{k\to \infty}\overline{R}(\omega_{k},r_{k}).$$
The directed singular value functions corresponding to these types of rectangles are defined as follows:
If $r_1\leq r_2$, let
\begin{align*}
\Phi^t(r)=\begin{cases}
r_2^t, &\text{ if } t\in [0,2],\\
r_1^{t-2}r_2^{2}, &\text{ if } t\in [2,2n+2],
\end{cases}
\end{align*}
and if $r_1\geq r_2$ let
\begin{align*}
\Phi^t(r)=\begin{cases}
r_1^t, &\text{ if } t\in [0,2n+1],\\
r_1^{2(2n+1)-t}r_2^{2(t-(2n+1))}, &\text{ if } t\in [2n+1,2n+2].
\end{cases}
\end{align*}

The arguments presented in this paper can be used to show that given a sequence $\underline{r}=(r_k)$, it is true that 
\begin{align*}
\dim_{\mathrm{H}} E_{\underline{r}}(\omega)=\inf \left\{ t\ \Big| \ \sum_{k}\Phi^{t}(r_k)<\infty \right\} \wedge (2n+2)
\end{align*}
$\mathbb{P}$-almost surely.
In the case $n=1$, this formula 
reduces to \cite[Theorem~1.1]{ekstrometal}.
\end{remark}

\newpage

\addcontentsline{toc}{section}{\refname}

\bibliographystyle{acm}
\bibliography{lahteet}

\end{document}